\documentclass[11pt]{amsart}
\usepackage{geometry}                
\usepackage{graphicx}
\usepackage{amssymb,hyperref}
\usepackage{epstopdf}
\usepackage[usenames]{color}
\usepackage[alphabetic]{amsrefs}

\title{Strong coprimality and strong irreducibility of Alexander polynomials}
\author{Evan M. Bullock and Christopher William Davis}
\date{}                                           

\newtheorem{theorem}{Theorem}[section]
\newtheorem{lemma}[theorem]{Lemma}

\newtheorem{proposition}[theorem]{Proposition}
\newtheorem*{proposition*}{Proposition}
\newtheorem{corollary}[theorem]{Corollary}

\newtheorem*{theorem*}{Theorem}
\newtheorem*{definition*}{Definition}
\theoremstyle{remark}
\newtheorem{remark}[theorem]{Remark}

\newcommand{\Z}{\mathbb{Z}}
\newcommand{\Q}{\mathbb{Q}}

\newcommand{\C}{\mathcal{C}}
\newcommand{\RC}{\mathcal{RC}}

\newcommand{\F}{\mathcal{F}}

\newcommand{\sign}{\operatorname{sign}}
\newcommand{\lnk}{\operatorname{lnk}}

\newcommand{\eref}[1]{(\ref{#1})}
\newcommand{\fun}[3]{#1\!\!:#2\rightarrow #3}

\begin{document}
\maketitle
\begin{abstract}  A polynomial $f(t)$ with rational coefficients is \emph{strongly irreducible} if $f(t^k)$ is irreducible for all positive integers $k$.   Likewise, two polynomials $f$ and $g$ are \emph{strongly coprime} if $f(t^k)$ and $g(t^l)$ are relatively prime for all positive integers $k$ and $l$.  We provide some sufficient conditions for strong irreducibility and prove that the Alexander polynomials of twist knots are pairwise strongly coprime and that most of them are strongly irreducible.  We apply these results to describe the structure of the subgroup of the rational knot concordance group generated by the twist knots and to provide an explicit set of knots which represent linearly independent elements deep in the solvable filtration of the knot concordance group.
\end{abstract}

\section{Introduction}

A knot is an oriented locally flat embedding of $S^1$ into $S^3$.  Modulo slice knots, i.e. knots which bound a locally flat embedding of the 2-disk in the 4-ball, the set of knots forms a group.  This group is called the \emph{knot concordance group} and is denoted $\C$.  In \cite{Le}, Levine defines a surjection from $\C$ to $\Z^\infty\oplus (\Z/2\Z)^\infty \oplus (\Z/4\Z)^\infty$.  Knots in the kernel of this map are called algebraically slice.  The quotient of the knot concordance group by algebraically slice knots is called the algebraic concordance group.  

Levine's work also shows that the Alexander polynomial of a knot, $\Delta_K(t)\in \Z[t]$, fits very well into the theory of algebraic concordance.  If a knot has irreducible Alexander polynomial then it is not algebraically slice.  Moreover, if a pair of knots $J$ and $K$ have coprime Alexander polynomials, then $J\#K$ is algebraically slice if and only if both of $J$ and $K$ are algebraically slice.

Moreover, in \cite{Cha07}, Cha discusses the \emph{rational concordance group} of knots, which we will denote by $\RC$.  A knot is called \emph{rationally slice} if it bounds a disk in a 4-manifold with the rational homology of a ball.  The appropriate replacements for irreducibility and coprimality of  Alexander polynomials in this setting are stronger analogues:  a polynomial $f(t)\in \Z[t]$ is called \textit{strongly irreducible} if $f(t^k)$ is irreducible for all positive integers $k$.  Two polynomials $f(t)$ and $g(t)$ are called \textit{strongly coprime} if for all positive integers $k$ and $l$, the polynomials $f(t^k)$ and $g(t^l)$ are coprime.  

The work of \cite{CHL10}, for example Theorem 7.7, shows that if one wishes to distinguish concordance classes of knots in more subtle cases, such as knots that are algebraically slice, then the notions of strong irreducibility and strong coprimality arise.  The difficulty of actually performing computations in this setting poses a hurdle in the application of their work.   For example, they are able to produce generating sets for infinite rank free abelian groups which sit deeply in the solvable filtration of the knot concordance group, but cannot prove that any one of these generators does not represent the zero element of the concordance group.  In section~\ref{infection}, we apply the algebraic tools of this paper to produce explicit \textit{linearly independent} sets of knots.

In \cite{Cha07}, Cha defines an epimorphism from $\RC$ to $\Z^\infty\oplus (\Z/2\Z)^\infty\oplus (\Z/4\Z)^\infty$.  The quotient of $\RC$ by the kernel of this epimorphism is called the \emph{rational algebraic concordance group}.  In section~\ref{rat conc}, we discuss the manner in which strong coprimality and strong irreducibility fit into this theory.  Specifically we prove that if two knots, $J$ and $K$ have strongly coprime Alexander polynomials, then $J\#K$ is rationally algebraically slice if and only if $J$ and $K$ are rationally algebraically slice.

In Section~\ref{strong coprimality} we prove Theorem~\ref{examples of strong coprimality}, showing that the twist knots, depicted in Figure~\ref{fig:twist knot}, have pairwise strongly coprime Alexander polynomials.  We provide an application to the structure of the group they generate in the rational algebraic concordance group.

\begin{figure}[htbp]
\setlength{\unitlength}{1pt}
\begin{picture}(150,100)
\put(15,0){\includegraphics[width=.25\textwidth]{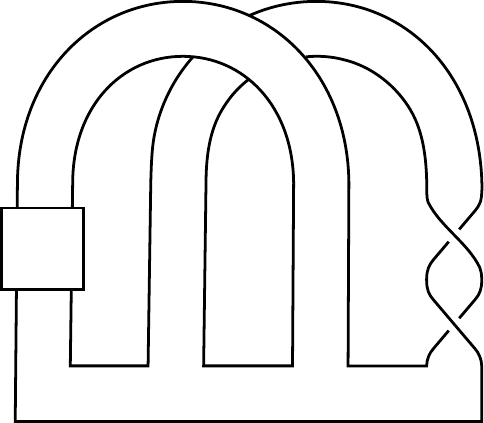}}
\put(21,36){$n$}
\end{picture}
\caption{The $n$-twist knot, $T_n$.}\label{fig:twist knot}
\end{figure}

\begin{theorem*}[Theorem \ref{examples of strong coprimality}]
For all positive integers $m \neq n$,
the Alexander polynomials $\Delta_{T_{n}}$ and $\Delta_{T_{m}}$ are strongly coprime.
\end{theorem*}

It is known that $\Delta_{T_{n}}$ is reducible precisely when $n=y(y+1)$ with $y\in \Z$; in fact, $n$ is of this form if and only if $T_n$ is algebraically slice (see \cite{CG86}).  We go on in Section~\ref{strong irreducibility} to prove that most of the twist knots have strongly irreducible Alexander polynomials.

\begin{theorem*}[Corollary~\ref{twist knots strongly irreducible}]
For every positive integer $n$ that is not a perfect power and not of the form $y(y+1)$ with $y\in \Z$, the Alexander polynomial $\Delta_{T_{n}}$ is strongly irreducible.  When $n$ is a perfect square, $\Delta_{T_n}$ is not strongly irreducible.  \end{theorem*}

In order to prove this theorem we develop the following sufficient conditions for strong irreducibility.

\begin{theorem*}[Corollary \ref{degnsimpleeisenstein}]
 Let $f=c_dt^d+\cdots+c_0\in \Z[t]$ be an irreducible polynomial of degree $d$, where the coefficients do not all share a common factor.  Then if some prime $p$ divides $c_d$ or $c_0$ exactly once, $f$ is strongly irreducible.
\end{theorem*}

\begin{theorem*}[Corollary \ref{strong irreducibility several primes corollary}]  Let $f(t)=c_dt^d+\cdots+c_1 t+c_0\in\Z[x]$, where $c_0$ and $c_1$ are relatively prime non-zero integers.  If $f$ is irreducible and $c_0\not=\pm \alpha^k$ for any integer $\alpha$ and natural number $k>1$, then $f$ is strongly irreducible.
\end{theorem*}


\section{applications to rational knot concordance}\label{rat conc}

%

In \cite{Cha07}, Jae Choon Cha defines the rational algebraic concordance group of knots in an analogous manner to the definition of the algebraic concordance group in \cite{Le}.  Cha defines a complete set of invariants for the rational algebraic concordance group by taking direct limits of Levine's complete set of invariants of algebraic concordance.  In this section we prove that these invariants have a splitting property for knots with strongly coprime Alexander polynomials.

An algebraic number $z$ is called \emph{reciprocal} if $z$ and $z^{-1}$ are roots the same irreducible polynomial over $\Q$.  Levine's invariants of algebraic concordance $s$, $e$ and $d$ are defined as follows:
\begin{itemize}
\item For $z$ a reciporical number with $|z|=1$, $s_z(K)\in \Z$ is the jump of the Tristram-Levine signature function at $z$.
\item For $z$ a reciprocal number, $e_z(K)\in \Z/2\Z$ is the number of times that the irreducible polynomial of $z$ divides the Alexander polynomial of $K$, reduced mod 2.
\item For $z$ a reciprocal number, $d_z(K)\in\frac{\Q(z+z^{-1})^\times}{\{u\overline u|~u\in \Q(z)^\times\}}$ is the discriminant of the $z$-primary part of a Seifert matrix of $K$.
\end{itemize}

All of these invariants have the property that they vanish when the Alexander polynomial of $K$ is relatively prime to the irreducible polynomial of $z$ \cite[Proposition 3.6 (1)]{Cha07}.  The invariants $s$ and $e$ are additive under connected sum, while $d_z(K\#J)=(-1)^{e_z(K)e_z(J)}d_z(K)d_z(J)$ \cite[Proposition 3.6 (4)]{Cha07}.

Let $P$ denote the set of sequences $(a_k)_{k=1}^\infty$ of reciporical numbers such that $(a_{nk})^n=a_k$ for all $n, k$.  Let $P_0$ be the subset of $P$ given by adding the restriction that $|a_k|=1$ for all $k$.  For $a=(a_k)\in P_0$ and $b=(b_k)\in P$, Cha defines $s_a(K)=(s_{a_k}(K))_{k=1}^\infty$, $e_b(K)=(e_{b_k}(K))_{k=1}^\infty$ and $d_b(K)=(d_{b_k}(K))_{k=1}^\infty$.  These form a complete set of invariants for rational algebraic concordance \cite[Theorem 3.13]{Cha07}.

Two polynomials $f(t), g(t)\in \Z[t]$ are called \textit{strongly coprime} if for every pair of integers $k$ and $l$, the polynomials $f(t^k)$ and $g(t^l)$ have no common roots.  We now prove a connection between rational algebraic concordance and the condition of strong coprimality.

\begin{proposition}\label{coprime to rat alg conc}
If two knots $J$ and $K$ have strongly coprime Alexander polynomials, then $J\#K$ is rationally algebraically slice if and only if both of $J$ and $K$ are.
\end{proposition}
\begin{proof}  It follows immediately from the additivity (up to sign) of $s$, $e$ and $d$ that if both $J$ and $K$ are rationally algebraically slice, then so is $J\#K$.  Assume now that $K$ is not rationally algebraically slice.  Then there exists some $a=(a_k)\in P$ such that one of $s_a(K)$, $e_a(K)$ or $d_a(K)$ is nonzero, so for some term $a_n$  in the sequence $a$, one of $s_{a_n}(K)$, $e_{a_n}(K)$ or $d_{a_n}(K)$ is nonzero.   This implies that $a_n$ is a root of $\Delta_K(t)$, the Alexander polynomial of $K$.  If one of  $s_{a}(J)$, $e_{a}(J)$ or $d_{a}(J)$ is nonzero then similarly $a_m$ is a root of $\Delta_J(t)$ for some $m$.  Then $(a_{mn})^m=a_n$ so that $a_{mn}$ is a root of $\Delta_K(t^m)$.  Similarly, $a_{mn}$ is a root of $\Delta_J(t^n)$, contradicting the assumption that $\Delta_K$ and $\Delta_J$ are strongly coprime.  Thus, it must be that $s_{a_n}(J)$, $e_{a_n}(J)$ and $d_{a_n}(J)$ vanish.

By additivity, we have $s_a(J\#K)=s_a(K)$,  $e_a(J\#K)=e_a(K)$, and  $d_a(J\#K)= d_a(K)$.  By assumption, one of these is nonzero, so we conclude that $J\#K$ is not rationally algebraically slice.
\end{proof}

\begin{corollary}\label{injection from algebraic to rational algebraic}
If knots $K_1, K_2,\dots $ have strongly irreducible Alexander polynomials that are distinct up to substitutions $f(t)\mapsto \pm f(t^k)$, then the map from the algebraic concordance group to the rational algebraic concordance group in injective on their span.
\end{corollary}
\begin{proof}
Since distinct strongly irreducible polynomials are strongly coprime, Proposition~\ref{coprime to rat alg conc} implies that the only way a linear combination $\underset{j=1}{\overset{n}{\#}}c_jK_j$ with $c_j\in \Z$ can be rationally algebraically slice is if $c_jK_j$ is rationally algebraically slice for all $j$.  

Now, we note that $K_j$ has infinite order in the algebraic concordance group if and only if $s_z(K_j)$ is nonzero for some reciprocal number $z$ with $|z|=1$.  By \cite[Proposition 6.1 (1)]{Cha07}, $z$ must be a root of $\Delta_{K_j}(t)$.  Let $a=(a_n)=(\sqrt[n]{z})$ be any compatible sequence of $n$th roots of $z$.  Then each $a_n$ is reciprocal, since $a_n$ and $a_n^{-1}$ are both roots of the same irreducible polynomial $\Delta_{K_j}(t^n)$.  Thus,  $a\in P_0$ and certainly $s_a(K_j)$ is nonzero, since its first entry is nonzero.  Since $s$ is a homomorphism to a torsion-free group it follows that $K_j$ is of infinite order in the rational algebraic concordance group.

The proof in the case that $K_j$ is of order two or four is the same, but uses the invariants $e$ and $d$ in place of $s$.
\end{proof}

%
%
%

\section{Strong coprimality}\label{strong coprimality}

Recall that the Alexander polynomial of the $n$-twist knot $T_n$ is \[\Delta_{T_{n}}(t) = nt^2-(2n+1)t+n.\]
In this section, we prove the following theorem:

\begin{theorem}\label{examples of strong coprimality}
For positive integers $m \neq n$,
the Alexander polynomials $\Delta_{T_{n}}$ and $\Delta_{T_{m}}$ are strongly coprime.
\end{theorem}
\begin{corollary}
If some linear combination $\underset{j=1}{\overset{m}{\#}}c_jT_j$ is rationally algebraically slice, then each $c_jT_j$ is rationally algebraically slice.
\end{corollary}
\begin{proof}This follows immediately from Theorem~\ref{examples of strong coprimality} and Proposition~\ref{coprime to rat alg conc}.
\end{proof}

\begin{proof}[Proof of Theorem~\ref{examples of strong coprimality}]

By the quadratic formula, the roots of $\Delta_{T_{n}}(t)$ are given by $r_n$ and $1/r_n$, where $r_{n} = \frac{2n+1+\sqrt{4n+1}}{2n}$.

If these polynomials were not strongly coprime, then for nonzero integers $k,l$, \begin{equation}\label{isogeny}
r_{n}^k = r_{m}^l.
\end{equation}

Since $r_n>1$, it must be that $\sign(k)=\sign(l)$, from here on we assume both are positive.  If $k$ and $l$ had a common factor $d$, we could take the positive real $d$th root of both sides of this equation.   We thus may assume that $k$ and $l$ are relatively prime.  The proof now proceeds by cases.

\begin{lemma}
If $(4n+1)(4m+1)$ is not the square of an integer, then $\Delta_{T_{n}}$ and $\Delta_{T_{m}}$ are strongly coprime.
\end{lemma}
\begin{proof}
Otherwise, (\ref{isogeny}) holds so that $r_{n}^k = r_{m}^l$.  An easy inductive argument shows that for each $k>0$ there are positive rationals $a$ and $b$ such that $\left(\frac{2n+1+\sqrt{4n+1}}{2n}\right)^k = a+b\sqrt{4n+1}$.  Thus, there are positive rational numbers $a,b,c,d$ such that
\begin{equation*}
a+b\sqrt{4n+1}=c+d\sqrt{4m+1},
\end{equation*} 
and rearranging this equation we see that 
\begin{equation*}
b\sqrt{4n+1}-d\sqrt{4m+1}=c-a.
\end{equation*}
By squaring both sides of this equation and performing arithmetic we get 
\begin{equation*}
b^2(4n+1)+d^2(4m+1)-2bd\sqrt{(4n+1)(4m+1)}=(c-a)^2
\end{equation*}
so that since $bd$ is a nonzero rational number, $\sqrt{(4n+1)(4m+1)}$ is rational, contradicting the assumption that $(4n+1)(4m+1)$ is not a square.
\end{proof}

Thus, we are reduced to the case that $(4n+1)(4m+1)$ is a square.  This in particular implies that there are odd integers $a,b,D$, with $D$ squarefree and congruent to 1 mod 4 such that  $4n+1=a^2D$ and $4m+1=b^2D$.  Making these substitutions, 

\begin{equation}\label{isogeny2}
r_{n} = \frac{a^2D-1+2(1+ a\sqrt{D})}{(a\sqrt{D}+1)(a\sqrt{D}-1)} = \frac{a\sqrt{D}+1}{a\sqrt{D}-1}, ~~\hspace{1cm}~~ r_{m} 
= \frac{b\sqrt{D}+1}{b\sqrt{D}-1}.
\end{equation}

We first deal with the case where $D=1$, i.e. where $r_n$ and $r_m$ are both rational numbers, namely

\begin{equation*}r_n = 
\frac{a+1}{a-1}=\frac{2y+2}{2y}=\frac{y+1}{y},\end{equation*}
where $a=2y+1$, and $r_m=\frac{z+1}{z}$ with $b=2z+1$.  We thus have 
\[r_n^k = r_{y(y+1)}^k =  \left(\frac{y+1}{y}\right)^k = \left(\frac{z+1}{z}\right)^l = r_{z(z+1)}^l = r_m^l,\]
and comparing prime factorizations, we see that we must have $y=\alpha^l$ and $z=\alpha^k$ for some natural number $\alpha>1$.  But then if $k<l$ we have
\[\left(\frac{y+1}{y}\right)^k = \left(1+\frac{1}{\alpha^l}\right)^k < \left(1+\frac{1}{\alpha^k}\right)^k < \left(1+\frac{1}{\alpha^k}\right)^l =  \left(\frac{z+1}{z}\right)^l,\]
a contradiction.

The proof in the case $D>1$ works exactly the same way, but since $r_n$ and $r_m$ are elements of $\Q(\sqrt{D})$ and instead of $\Q$, we must replace unique prime factorization in $\Z$ with unique factorization into prime ideals (see \cite{algebraicnumbertheory} II.1, p. 54) in the ring  $\Z\left[\frac{1+\sqrt{D}}{2}\right]$ of integers   of the number field $\Q(\sqrt{D})$.

To this end, let $p$ be a prime number dividing $n$.  Then the ideal generated by $p$ in $\Z\left[\frac{1+\sqrt{D}}{2}\right]$ is not prime.  It factors as $
\langle p\rangle = P_1 P_{-1},$ where 
\begin{equation*}
P_1=\left\langle p,\frac{a\sqrt{D}+1}{2}\right\rangle,~~~ P_{-1}=\left\langle p,\frac{a\sqrt{D}-1}{2}\right\rangle
\end{equation*}

In order to see this, we expand $P_1 P_{-1} = \left\langle p^2, p\frac{a\sqrt{D}+1}{2}, p\frac{a\sqrt{D}-1}{2}, \frac{a^2D-1}{4}\right\rangle.$  The first, second and third of these generators are clearly in $\langle p\rangle$.  The fourth, $\frac{a^2D-1}{4}=n$ is a multiple of $p$ by design.  Thus, $\langle p\rangle\supseteq P_1P_{-1}$.  in order to see the opposite containment, note that the difference between the second and third of the generators is exactly $p$.  

Recall that for an ideal $I$ in $\Z\left[\frac{1+\sqrt{D}}{2}\right]$, the norm $N(I)$ is defined to be the order of the quotient $\left.\Z\left[\frac{1+\sqrt{D}}{2}\right]\right/I$ as an abelian group.  We will require the following basic properties of the norm (see \cite{algebraicnumbertheory} pp. 56-57): for ideals $I,J\subseteq \Z\left[\frac{1+\sqrt{D}}{2}\right]$,  
\begin{itemize}
\item if $I=\langle r+s\sqrt{D}\rangle$ is principal, then $N(I)=\left|N\left(r+s\sqrt{D}\right)\right| = \left|r^2-Ds^2\right|$,
\item if $I\supsetneq J$, then $N(I)$ is a proper divisor of $N(J)$,
\item if $N(I)$ is a prime integer, then $I$ is a prime ideal, and
\item the norm is multiplicative: $N(IJ)=N(I)N(J)$.
\end{itemize}
Thus, since \[N(P_1)N(P_{-1}) = N(\langle p\rangle) = p^2\]
with $N(P_1)<p^2$ and $N(P_{-1})<p^2$, we must have $N(P_1)=N(P_{-1})=p$ and both $P_1$ and $P_{-1}$ are prime ideals.

\begin{lemma}\label{multiplicity lemma}
If $p^k$ divides $n$ then for $\epsilon=\pm 1$, 
\[P_{\epsilon}^k=\left\langle p^k, \frac{a\sqrt{D}+\epsilon}{2}\right\rangle.\]
\end{lemma}
\begin{proof}
We proceed by induction; the lemma holds for $k=1$ by the definition of $P_{\epsilon}$.   If it holds for $k-1$, then 
\begin{equation*}
\begin{array}{rcl}
P_{\epsilon}^k&=&P_{\epsilon}^{k-1}P_{\epsilon}\\
&=&\left\langle p^{k-1}, \frac{a\sqrt{D}+\epsilon}{2}\right\rangle\left\langle p, \frac{a\sqrt{D}+\epsilon}{2}\right\rangle\\
&=&\left\langle p^{k}, p\frac{a\sqrt{D}+\epsilon}{2}, n+\epsilon\frac{a\sqrt{D}+\epsilon}{2}\right\rangle,
\end{array}
\end{equation*}
and since $p^k$ divides $n$, we have
\begin{equation*}
\begin{array}{rcl}
P_{\epsilon}^k&=&\left\langle p^{k}, p\frac{a\sqrt{D}+\epsilon}{2}, \frac{a\sqrt{D}+\epsilon}{2}\right\rangle\\
&=&\left\langle p^{k}, \frac{a\sqrt{D}+\epsilon}{2}\right\rangle.
\end{array}
\end{equation*}
completing the proof.
\end{proof}

Returning to the situation of interest, we determine the multiplicity with which $P_{1}$ divides the numerators and denominators of \eref{isogeny2} in terms of the multiplicity with which $p$ divides $m$ and $n$.

\begin{lemma}\label{divisibility lemma}
Assume that $p$ divides $n$ with multiplicity $k>0$ over the integers, then 
$\frac{a\sqrt{D}+1}{2}$ is contained in $P_1^k$ but not $P_1^{k+1}$ and $\frac{a\sqrt{D}-1}{2}$ is not contained in $P_1$.
\end{lemma}
\begin{proof}

First, note that if $\frac{a\sqrt{D}-1}{2}\in P_1$ then since $\frac{a\sqrt{D}+1}{2}\in P_1$, it would follow that $1\in P_1$, contradicting that $P_1$ is a proper ideal.

By Lemma~\ref{multiplicity lemma}, $\frac{a\sqrt{D}+1}{2}\in P_{1}^k$.  
In order to show that no greater power of $P_1$ contains $\frac{a\sqrt{D}+1}{2}$, we perform a norm computation:  
\[N\left(\left\langle\frac{a\sqrt{D}+1}{2}\right\rangle\right)=\left|\frac{(1+a\sqrt{D})(1-a\sqrt{D})}{4}\right|=n.\] 
Recalling that $N(P_1)=p$, we note that if $\frac{a\sqrt{D}+1}{2}$ were contained in $ P_{1}^{k+1}$ then it would follow that $p^{k+1}$ divides $n$, in contradiction to the assumption to the contrary.  This completes the proof.
\end{proof}

Now suppose that (\ref{isogeny}) holds, so that in light of \eref{isogeny2}:

\begin{equation}\label{cancellation}
\begin{array}{cc}
\left(\left.{\frac{a\sqrt{D}+1}{2}}\right/{\frac{a\sqrt{D}-1}{2}}\right)^k=
\left(\left.{\frac{b\sqrt{D}+1}{2}}\right/{\frac{b\sqrt{D}-1}{2}} \right)^l
\end{array}
\end{equation}

Cross multiplying reduces this to
\begin{equation}\label{isogeny3}
\begin{array}{c}
\left(\frac{a\sqrt{D}+1}{2}\right)^k\left(\frac{b\sqrt{D}-1}{2}\right)^l = \left(\frac{a\sqrt{D}-1}{2}\right)^k\left(\frac{b\sqrt{D}+1}{2}\right)^l.
\end{array}
\end{equation}

First, suppose that $p$ is a prime factor of $n$ but is not a factor of $m$.  By Lemma~\ref{divisibility lemma}, $\frac{a\sqrt{D}-1}{2}\not\in P_1$.  Since the norm of $\frac{b\sqrt{D}-1}{2}$ is $m$ which is not divisible by $p=N(P_1)$, we must have $\frac{b\sqrt{D}-1}{2}\not\in P_1$.  
Thus, the right-hand side of \eref{isogeny3} is not an element of the prime ideal $P_1$, but we know the left-hand side is because $\frac{a\sqrt{D}+1}{2}\in P_1$, a contradiction.  Thus, $n$ and $m$ must have the same set of prime factors.

Suppose then that $p$ is a prime factor of both $n$ and $m$.  Let $x=v_p(n)$ be the highest power such that $p^x$ divides $n$ and likewise let $y=v_p(m)$ be the highest power of $p$ such that $p$ divides $m$.  We have two factorizations of the ideal $\langle p\rangle$ in the Dedekind domain $\Z\left[\frac{1+\sqrt{D}}{2}\right]$: 
\begin{equation*}
\begin{array}{rclll}
\langle p\rangle&=&P_1P_{-1}&\text{ where }&P_1=\left\langle p,\frac{a\sqrt{D}+1}{2}\right\rangle,~~P_{-1}=\left\langle p,\frac{a\sqrt{D}-1}{2}\right\rangle,\text{~~and}\\
\langle p\rangle&=&\widetilde{P}_{1}\widetilde{P}_{-1}&\text{ where }&\widetilde{P}_1=\left\langle p,\frac{b\sqrt{D}+1}{2}\right\rangle,~~\widetilde{P}_{-1}=\left\langle p,\frac{b\sqrt{D}-1}{2}\right\rangle.
\end{array}
\end{equation*}
By unique factorization of ideals, either $P_1 = \widetilde{P}_1$ or $P_1=\widetilde{P}_{-1}$.  In the latter case, by Lemma~\ref{divisibility lemma}, $P_1$ divides the left hand side of \eref{isogeny3} with multiplicity $xk+yl$ and the right hand side with multiplicity $0$, which is a contradiction.  Thus, $P_1=\widetilde{P}_1$ and $P_1$ divides the left hand side with multiplicity $xk$ and the left hand side with multiplicity $yl$.  It follows that $xk=yl$.  


Since we can assume that $k$ and $l$ are relatively prime, $xk=yl$ implies that there is some integer $c$ (depending on $p$) with $y=ck$ and $x=cl$.  Thus, $n$ and $m$ have related prime factorizations:
\begin{equation}
\begin{array}{rclllll}
n&=&p_1^{c_1l}p_2^{c_2l}\dots p_f^{c_fl}&=&(p_1^{c_1}p_2^{c_2}\dots p_f^{c_f})^l&=&\alpha^l,\text{ and}\\
m&=&p_1^{c_1k}p_2^{c_2k}\dots p_f^{c_fk}&=&(p_1^{c_1}p_2^{c_2}\dots p_f^{c_f})^k&=&\alpha^k,\\
\end{array}
\end{equation} 
for some integer $\alpha=p_1^{c_1}p_2^{c_2}\dots p_f^{c_f}>1$.

Now, having analyzed all the finite primes, we consider the ``infinite prime'', i.e. we show using an inequality that $(r_n)^k=(r_m)^l$ is impossible over the real numbers
for $n=\alpha^l$ and $m=\alpha^k$.
Let $\alpha$ be any integer greater than $1$.  Exchanging $n$ and $m$ if necessary, we may assume $k<l$.  Substituting $n=\alpha^l$ and $m=\alpha^k$ in to the original expressions for $r_n$ and $r_m$ from the quadratic formula, we get
\begin{equation*}
\left(\frac{2\alpha^l+1+\sqrt{4\alpha^l+1}}{2\alpha^l}\right)^k=
\left(\frac{2\alpha^k+1+\sqrt{4\alpha^k+1}}{2\alpha^k}\right)^l.
\end{equation*}
  Making a simplification, this is equivalent to
  \begin{equation*}
\left(1+\frac{1}{2\alpha^l} + \sqrt{\frac{1}{\alpha^l}+\frac{1}{4\alpha^{2l}}}\right)^k=
\left(1+\frac{1}{2\alpha^k} + \sqrt{\frac{1}{\alpha^k}+\frac{1}{4\alpha^{2k}}}\right)^l.
\end{equation*}

Observe that for $l>0$, the real number $1+\frac{1}{2\alpha^l} + \sqrt{\frac{1}{\alpha^l}+\frac{1}{4\alpha^{2l}}}$ is greater than $1$ and that this number increases as  $l$ decreases.  Thus, since $k<l$,
\begin{equation*}
\left(1+\frac{1}{2\alpha^l} + \sqrt{\frac{1}{\alpha^l}+\frac{1}{4\alpha^{2l}}}\right)^k<
\left(1+\frac{1}{2\alpha^k} + \sqrt{\frac{1}{\alpha^k}+\frac{1}{4\alpha^{2k}}}\right)^k<
\left(1+\frac{1}{2\alpha^k} + \sqrt{\frac{1}{\alpha^k}+\frac{1}{4\alpha^{2k}}}\right)^l.
\end{equation*}
In particular, this implies that these are not equal and completes the proof of Theorem~\ref{examples of strong coprimality}.
\end{proof}

\section{Strong irreducibility}\label{strong irreducibility}

A polynomial $f\in \Q[t]$ is \emph{strongly irreducible} if $f(t^k)$ is irreducible for all positive integers $k$.  In this section, we apply the idea behind Eisenstein's Irreducibility Criterion to prove various criteria that guarantee strong irreducibility.   Recall (\cite{Gou}, 2.1) that for $p\in \Z$ prime, the \emph{$p$-adic valuation} on $\Q$ is the function $\fun{v_p}{\Q}{\Z\cup\{+\infty\}}$ defined by 
\[v_p(x) = \begin{cases}k & x= \frac{r}{s}\cdot p^k \text{~\hspace{.25cm}where $k,r,s\in \Z$ satisfy $p\nmid r$, ~~$p\nmid s$}, \\ +\infty & x=0.\end{cases}
\]

\begin{theorem}[An Eisenstein criterion for strong irreducibility]\label{Eisenstein criterion for strong irreducibility} Let $f=c_dt^d+\cdots+c_0$ be an irreducible polynomial in $\Q[t]$.  Let $p$ be a prime number and suppose
that for some $0\leq i<j\leq d$, the following conditions are satisfied, where $a=j-i$, $b=v_p(c_j)-v_p(c_i)$, and $m=b/a$:
\begin{enumerate}
\item\label{firstcondition} $c_i$ and $c_j$ are non-zero,
\item\label{nothorizontal} $v_p(c_i)\neq v_p(c_j)$,
\item\label{nolatticepoint} $a$ and $b$ are relatively prime,
\item\label{notwild} $p \nmid b$, and
\item\label{convex hull condition} $\begin{array}{ll}
v_p(c_k) \geq m(k-i)+v_p(c_i) &\text{for $i<k<j$, and } \\
v_p(c_k) > m(k-i)+v_p(c_i)&\text{for $0\leq k<i$ or $j<k\leq d$. }
\end{array}$
\end{enumerate}
Then if $f(t^{|b|})$ is irreducible, $f$ is strongly irreducible.
\end{theorem}

We defer the proof of Theorem \ref{Eisenstein criterion for strong irreducibility} until the end of the section.  To make sense of the hypotheses, and especially condition (\ref{convex hull condition}), we recall the definition of the \emph{Newton polygon of $f$ at the prime $p$} (see \cite{Gou}, 6.4): we plot the points $(k,v_p(c_k))$ in the plane for $0\leq k\leq d$, and take the lower part of the boundary of their convex hull.\footnote{We omit the points $(k,+\infty)$ arising from coefficients $c_k=0$, since these points would not affect the lower boundary of the convex hull.}  For example, if $f(t)=8t^4 -26t^3+ 35t^2 -26t+ 8$ is the Alexander polynomial of the knot $12a_{1163}$ from \cite{knotinfo}, then the Newton polygon of $f$ at the primes $p=2$ and $p=13$ are the following:

\setlength{\unitlength}{1.2cm}
\begin{picture}(6, 5)(-1,-1)
\thinlines
\multiput(0,-.5)(1,0){5}{\line(0,1){4}}
\multiput(-.5,0)(0,1){4}{\line(1,0){5}}
\thicklines
\put(0,-.5){\line(0,1){4}}
\put(-.5,0){\line(1,0){5}}
\put(4.35,-.35){\small $k$}
\put(-.68,3.35){\tiny $v_2(c_k)$}

\put(0,3){\circle*{.15}}
\put(0,3){\line(1,-2){1}}
\put(1,1){\circle*{.15}}
\put(1,1){\line(1,-1){1}}
\put(2,0){\circle*{.15}}
\put(2,0){\line(1,1){1}}
\put(3,1){\circle*{.15}}
\put(3,1){\line(1,2){1}}
\put(4,3){\circle*{.15}}
\end{picture}
\setlength{\unitlength}{1.2cm}
\begin{picture}(6, 5)(-1,-1)
\thinlines
\multiput(0,-.5)(1,0){5}{\line(0,1){4}}
\multiput(-.5,0)(0,1){4}{\line(1,0){5}}

\thicklines
\put(0,-.5){\line(0,1){4}}
\put(-.5,0){\line(1,0){5}}
\put(4.35,-.35){\small $k$}
\put(-.78,3.35){\tiny $v_{13}(c_k)$}

\linethickness{.065\unitlength}

\put(0,0){\circle*{.15}}
\put(0,0){\line(1,0){4}}
\put(1,1){\circle*{.15}}
\put(2,0){\circle*{.15}}
\put(3,1){\circle*{.15}}
\put(4,0){\circle*{.15}}
\end{picture}

In terms of the Newton polygon for $f$ at $p$, we can restate conditions (\ref{firstcondition})-(\ref{convex hull condition}) as saying that the Newton polygon has some edge which (\ref{nothorizontal}) is  not horizontal, (\ref{nolatticepoint}) does not pass through a lattice point besides its endpoints, and (\ref{notwild}) has a vertical height $b$ that is not divisible by $p$.  In the case of the above example, the edge $((1,1),(2,0))$ of the Newton polygon at $p=2$ satisfies conditions (\ref{firstcondition})-(\ref{convex hull condition}) with $b=1$.  Thus, since $f$ is irreducible, by Theorem \ref{Eisenstein criterion for strong irreducibility}, $f$ is strongly irreducible.  Since $12a_{1163}$ is of infinite order in the algebraic concordance group, Corollary~\ref{injection from algebraic to rational algebraic} implies that it is of infinite order in the rational algebraic concordance group.

\begin{corollary}\label{degnsimpleeisenstein} Let $f=c_dt^d+\cdots+c_0\in \Z[t]$ be an irreducible polynomial of degree $d$, where we assume the coefficients do not all share a common factor.  Then if some prime $p$ divides $c_d$ or $c_0$ exactly once, $f$ is strongly irreducible.
\end{corollary}
\begin{proof} Suppose that $p$ divides $c_0$ exactly once and that $j$ is the first index for which $c_j$ is not divisible by $p$.  
\begin{center}
\setlength{\unitlength}{1.2cm}
\begin{picture}(6, 5)(-1,-1)
\thinlines
\multiput(0,-.5)(1,0){5}{\line(0,1){4}}
\multiput(-.5,0)(0,1){4}{\line(1,0){5}}
\thicklines
\put(0,-.5){\line(0,1){4}}
\put(-.5,0){\line(1,0){5}}
\put(4.35,-.35){\small $k$}
\put(-.68,3.35){\tiny $v_p(c_k)$}
\put(3.05,-0.2){\tiny $(j,0)$}
\put(.05,1.1){\tiny $(0,1)$}

\put(0,1){\circle*{.15}}
\put(0,1){\line(3,-1){3}}
\put(3,0){\circle*{.15}}
\put(3,0){\line(6,1){1.5}}
\put(1,1){\circle*{.15}}
\put(2,3){\circle*{.15}}

\put(4,2){\circle*{.15}}
\end{picture}
\end{center}
Then $((0,1),(j,0))$ is an edge of the Newton polygon for $f$ at $p$, with $a=j$, $b=-1$, so certainly $a$ and $b$ are relatively prime, $p\nmid b$, and the hypotheses of Theorem \ref{Eisenstein criterion for strong irreducibility} are satisfied.  We conclude $f$ is strongly irreducible.  

Similarly, in the case that instead $p$ divides $c_d$ exactly once, we consider the edge $((i,0),(d,1))$, where $i$ is the last index for which $c_i$ is not divisible by $p$.
\end{proof}

Note that the previous strong irreducibility result of Jae Choon Cha follows immediately from Corollary \ref{degnsimpleeisenstein}:
\begin{proposition*}[\cite{Cha07}, Proposition 3.18]
Suppose $\lambda(t)=pt^2-(2p+k)t+p$, where $p$ is a prime and $k$ is an integer such that $k\not\equiv 0 \pmod{p}$ and $k\not\equiv -2p\pm 1 \pmod{p^2}$.  Then $\lambda(t)$ is strongly irreducible.
\end{proposition*}
We need the second congruence condition only to guarantee that $\pm \frac{1}{p}$, $\pm p$ are not the roots of $\lambda(t)$, and that therefore $\lambda(t)$ is irreducible over $\Q$.  Indeed, this means that the hypothesis $k\not\equiv -2p\pm 1 \pmod{p^2}$ in the proposition can be weakened to $k\neq -2p\pm(1+p^2)$.

We now state a second strong irreducibility criterion that takes into account data at several primes.
\begin{theorem}\label{strong irreducibility theorem using several primes}
Let $f=c_dt^d+\cdots+c_0 \in \Q[x]$ be irreducible.  Assume that for every prime $q$ there exists a prime $p$ and an edge $( (i,v_p(c_i)),(i+a,v_p(c_i)+b) )$  of the Newton polygon for $f$ at $p$ such that the numbers $a$ and $b$ are relatively prime and $q\nmid b$.  Then $f$ is strongly irreducible.
\end{theorem}

Again, we defer the proof until the end of the section.

\begin{corollary}\label{strong irreducibility several primes corollary}  Let $f(t)=c_dt^d+\cdots+c_1 t+c_0\in\Z[x]$, where $c_0$ and $c_1$ are relatively prime non-zero integers.  If $f$ is irreducible and $c_0\not=\pm \alpha^k$ for any integer $\alpha$ and natural number $k>1$, then $f$ is strongly irreducible.
\end{corollary}
\begin{proof} Let $c_0=\pm p_1^{e_1}p_2^{e_2}\cdots p_r^{e_r}$ be the prime factorization of $n$.  Then for any prime $q$, since $c_0$ is not a $q$th power up to sign, there is some $i$ for which $q\nmid e_i$.  At the prime $p=p_i$, the Newton polygon for $f$ has an edge $((0,e_i),(1,0))$, for which $a=1$ so $a$ and $b$ are certainly relatively prime, and $b=-e_i$ is not divisible by $q$.  Thus, the hypotheses of Theorem \ref{strong irreducibility theorem using several primes} are satisfied and $f$ is strongly irreducible.
\end{proof}
As in Corollary \ref{degnsimpleeisenstein}, we could replace $c_0$ and $c_1$ with $c_d$ and $c_{d-1}$ in the statement of Corollary \ref{strong irreducibility several primes corollary}, either by giving the analogous proof or by replacing $f$ with $t^df(t^{-1})$.

\begin{corollary}\label{twist knots strongly irreducible}
If a natural number $n$ is not a perfect power and $n\neq y(y+1)$ for any $y\in \Z$, then the Alexander polynomial
\[\Delta_{T_{n}}(t) = nt^2-(2n+1)t+n \]
is strongly irreducible. 
\end{corollary}
\begin{proof} This follows immediately from Corollary \ref{strong irreducibility several primes corollary}: 
the $n\neq y(y+1)$ condition is equivalent to the irreducibility of $\Delta_{T_{n}}(t)$ over $\Q$ (see the analysis of the $D=1$ case in the proof of Theorem \ref{examples of strong coprimality}).
\end{proof}

\begin{remark}
In the case where $n=\alpha^2$ is a perfect square, the polynomial $\Delta_{T_{n}}(t)$ fails to be strongly irreducible:
\[\Delta_{T_{n}}(t^2) = \alpha^2 t^4 - (2\alpha^2+1)t^2+\alpha^2 = (\alpha t^2-t-\alpha)(\alpha t^2+t-\alpha).\] 

In fact, in this case, the $n$-twist knot is rationally algebraically slice.  In order to see this, we check that Cha's three invariants $s$, $e$, and $d$ all vanish.  If $s$ were nonzero then for some $a=(a_j)\in P_0$, the invariant $s_{(a)}(T_n)$ would not be identically zero.  Thus, $s_{a_j}(T_n)$ would be nonzero for some $a_j$, implying that $a_j$ is a root of $\Delta_{T_n}$.  Since $a\in P_0$, the term $a_{2j}$ is a square root of $a_j$, and hence must satisfy one of the irreducible factors of $\Delta_{T_n}(t^2)$.  But neither of these factors are symmetric, which implies that $a_{2j}$ is not reciprocal, contradicting that $(a)$ is in $P_0$.  The argument that $d$ and $e$ vanish is the same.

On the other hand, when $n$ is a perfect power but not a perfect square, we have found no examples where $\Delta_{T_{n}}(t)$ fails to be strongly irreducible, but it seems the techniques of this section can't easily be applied in this case in general.  

For example, when $n=6^3=216$, the polynomial $\Delta_{T_{216}}(t^3)$ has roots in both $\Q_2$ and $\Q_3$, i.e. in $\Q_p$ for every $p$ at which the Newton polygon is not simply a single horizontal segment.  On the other hand, $\Delta_{T_{216}}(t)$ is nevertheless strongly irreducible: applying Theorem \ref{Eisenstein criterion for strong irreducibility} to $\Delta_{T_{216}}(t)$ at $p=2$, we see that we need only check that $\Delta_{T_{216}}(t^3)$ is irreducible over $\Q$, which is easily checked by computer (in fact, it's irreducible mod 11).
\end{remark}

Applying Corollary~\ref{injection from algebraic to rational algebraic}, the above discussion yields the following:
\begin{corollary}\label{rational concordance of twist knots}
If $K$ is the kernel of the map from the subgroup generated by $\{T_n|~n>0\}$ in the algebraic concordance group
 to the rational algebraic concordance group, then $$\langle\{T_n|~n\text{ is a perfect square}\}\rangle\subseteq K \subseteq \langle\{T_n|~n\text{ is a perfect power}\}\rangle.$$
\end{corollary}
\par\noindent Numerical evidence suggests that $K=\langle\{T_n|~n\text{ is a perfect square}\}\rangle$.

\begin{proof}[Proof of Theorem \ref{Eisenstein criterion for strong irreducibility}]
Let $\Q(r)$ be the extension of $\Q$ obtained by adjoining a root $r$ of the irreducible polynomial $f$.  Then $[\Q(r):\Q]=d$, and to show that the degree $xd$ polynomial $f(t^x)$ is irreducible over $\Q$, it suffices to show that $[\Q(\sqrt[x]{r}):\Q(r)]=x$ for some $x$th root $\sqrt[x]{r}$ of $r$.  If $x=yz$, then 
\[ [\Q(\sqrt[x]{r}) : \Q(r)]=[\Q(\sqrt[x]{r}):\Q(\sqrt[y]{r})] [\Q(\sqrt[y]{r}) : \Q(r)], \]
 where $[\Q(\sqrt[x]{r}):\Q(\sqrt[y]{r})]\leq z$ and $[\Q(\sqrt[y]{r}) : \Q(r)]\leq y$.  Therefore, if $f(t^x)$ is irreducible, or equivalently $[\Q(\sqrt[x]{r}) : \Q(r)]=x$, then $[\Q(\sqrt[y]{r}) : \Q(r)]=y$ and $f(t^y)$ is irreducible as well.  
 
 Thus, to show $f$ is strongly irreducible, it suffices to show that $[\Q(\sqrt[x]{r}) : \Q(r)]=x$ when $x=|b|z$ is a multiple of $b$, and since we are assuming that $f(t^{|b|})$ is irreducible, it suffices to show that $[\Q(\sqrt[x]{r}) : \Q(\sqrt[|b|]{r})]=z$.

Now, applying standard results on Newton polygons (see \cite{Gou} 6.4, especially Theorem 6.4.7), we see that over the $p$-adics, $f$ has a factor $g\in\Q_p[t]$ of degree $a$, such that every root $\widetilde{r}\in\overline{\Q_p}$ of $g$ satisfies $v_p(\widetilde{r})=m=b/a$,
where the valuation $\fun{v_p}{\overline{\Q_p}}{\Q\cup\{+\infty\}}$ is the extension of $\fun{v_p}{\Q_p}{\Z\cup\{+\infty\}}$.  Since $f$ is irreducible over $\Q$ and $\overline{\Q_p}$ is algebraically closed, there exists an inclusion $\Q(r)\hookrightarrow \overline{\Q_p}$ taking $r$ to $\widetilde{r}$.  Thus, we identify $\Q(r)$ with a subfield of $\overline{\Q_p}$, where $r$ is a root of $g$.  

Now, since $a$ and $b$ are relatively prime and $v_p(r)=b/a$, the ramification index  $e\left(\Q_p(r)/\Q_p\right)$ is a multiple of $a$, but $r$ is a root of a polynomial $g$ of degree $a$.  Therefore $g$ is irreducible and $e\left(\Q_p(r)/\Q_p\right)=a$.  (For basic facts on ramified and unramified extensions and ramification index, see \cite{Gou} 5.4.)

Let $\pi$ be a uniformizer of $\Q_p(r)$.  Then $\Q_p(\sqrt[|b|]{r})=\Q_p\left(r, \sqrt[|b|]{r \pi^{-b}}\right)$, so since $v_p(r\pi^{-b})=0$ and $p\nmid b$, the extension 
$\Q_p(\sqrt[|b|]{r})/\Q_p(r)$ is unramified.  On the other hand, $v_p(\sqrt[x]{r})=\frac{m}{x}=\frac{b}{az|b|}=\pm \frac{1}{az}$, so $e\left(\Q_p(\sqrt[x]{r})/\Q_p\right)$ is a multiple of $az$.  Now, since
\[e\left(\Q_p(\sqrt[x]{r})/\Q_p\right)=e\left(\Q_p(\sqrt[x]{r})/\Q_p(\sqrt[|b|]{r})\right)\cdot e\left(\Q_p(\sqrt[|b|]{r})/\Q_p(r)\right)\cdot e\left(\Q_p(r)/\Q_p\right), \]
with $e\left(\Q_p(\sqrt[|b|]{r})/\Q_p(r)\right)=1$ and $e\left(\Q_p(r)/\Q_p\right)=a$, we find that $e\left(\Q_p(\sqrt[x]{r})/\Q_p(\sqrt[|b|]{r})\right)$ is a multiple of $z$.  Therefore 
$[\Q_p(\sqrt[x]{r}):\Q_p(\sqrt[|b|]{r})]$ is a multiple of $z$, and hence $[\Q(\sqrt[x]{r}):\Q(\sqrt[|b|]{r})]$ is a multiple of $z$, and since this extension is obtained by taking a $z$th root, we have $[\Q(\sqrt[x]{r}):\Q(\sqrt[|b|]{r})]=z$, as desired.
\end{proof}

\begin{proof}[Proof of Theorem \ref{strong irreducibility theorem using several primes}]
Let $\Q(r)$ be the extension of $\Q$ obtained by adjoining a root $r$ of the irreducible polynomial $f$.  As in the proof of Theorem \ref{Eisenstein criterion for strong irreducibility}, we must show that for every natural number $x$,  we have $[\Q(\sqrt[x]{r}):\Q(r)]=x$ for some $x$th root $\sqrt[x]{r}$ of $r$.

We observe that if $x=yz$, then $\Q(\sqrt[x]{r})$ contains subfields $\Q(\sqrt[y]{r})$ and $\Q(\sqrt[z]{r})$.  Thus, if $[\Q(\sqrt[y]{r}):\Q(r)]=y$ and $[\Q(\sqrt[z]{r}):\Q(r)]=z$, with $y$ and $z$ relatively prime, we could conclude that $[\Q(\sqrt[x]{r}):\Q(r)]=x$.  It therefore suffices to prove that $[\Q(\sqrt[x]{r}):\Q(r)]=x$ in the case where $x=q^h$ is a prime power.

Now, let $p$ be the prime associated to $q$ in the hypotheses of the theorem.  As in the proof of Theorem \ref{Eisenstein criterion for strong irreducibility}, we may conclude that over $\Q_p$, the polynomial $f$ has a factor $g$ whose roots $\widetilde{r}\in \overline{\Q_p}$ all satisfy $v_p(\widetilde{r})=b/a$.  Again, since $f$ is irreducible over $\Q$ we can identify $\Q(r)$ with a subfield of $\overline{\Q_p}$ in which $r$ is a root of $g$.  Again, since $a$ and $b$ are relatively prime, $g$ is irreducible over $\Q_p$ and $\Q_p(r)/\Q_p$ is totally ramified of degree $a$.

Now, since $v_p(\sqrt[q^h]{r}) = \frac{b}{aq^h}$ and $q\nmid b$, the ramification index $e\left(\Q_p(\sqrt[q^h]{r})/\Q_p\right)$ is a multiple of $aq^h$.  Therefore, since 
\[e\left(\Q_p(\sqrt[q^h]{r})/\Q_p\right) = e\left(\Q_p(\sqrt[q^h]{r})/\Q_p(r)\right) \cdot e\left(\Q_p(r)/\Q_p\right),\] 
and $e\left(\Q_p(r)/\Q_p\right)=a$, we conclude that $e\left(\Q_p(\sqrt[q^h]{r})/\Q_p(r)\right)$ is a multiple of $q^h$ and hence $[\Q_p(\sqrt[q^h]{r}):\Q_p(r)]$ and $[\Q(\sqrt[q^h]{r}):\Q(r)]$ are multiples of $q^h$.   Since $[\Q(\sqrt[q^h]{r}):\Q(r)]\leq q^h$, we conclude $[\Q(\sqrt[q^h]{r}):\Q(r)]=q^h$ as desired.
\end{proof}

\section{Application: making the \cite{CHL10} construction of knots deep in the solvable filtration explicit.}\label{infection}

 In \cite{COT03}, Cochran, Orr and Teichner define a filtration of the knot concordance group indexed by half-integers:
$$\dots\subseteq\F_{1.5}\subseteq\F_1\subseteq\F_{.5}\subseteq\F_0\subseteq \C.$$
In \cite{CHL10}, Cochran, Harvey and Leidy perform an iterated infection process along a particular class of ribbon knots to yield subgroups of $\F_n/\F_{n.5}$ isomorphic to $\Z^\infty$.  We recall their construction:

\begin{definition*}[Definition 7.2 of \cite{CHL10}]
Given a ribbon knot $R$ in $S^3$ and an unknotted curve $\eta$ in $S^3-R$ such that the linking number $\lnk(R,\eta)=0$, the pair $(R,\eta)$, also written $R_\eta$, is called a \emph{robust doubling operator} if:
\begin{enumerate}
\item the rational Alexander module of $R$, $A_0(R)$ is cyclic, generated by $\eta$, and $A_0(R)\cong \frac{\Q[t,t^{-1}]}{\delta(t)\delta(t^{-1})}$ for a prime polynomial $\delta$, and
\item for each isotropic submodule $P$ of $A_0(R)$, with the first order signature corresponding to $P$, $\rho(R,\phi_P)$, vanishes or $P$ corresponds to a ribbon disk for $R$.
\end{enumerate}
\end{definition*}

The justification for calling the pair $(R,\eta)$ an operator is a procedure called \emph{infection} which takes a doubling operator $R_\eta$ and a knot $J$ and produces a knot $R_\eta(J)$ by cutting the strands of $R$ which pass through a disk bounded by $\eta$ and tying them into the knot $J$, as is indicated in Figure~\ref{fig:infection}.

 \begin{figure}[htbp]
\setlength{\unitlength}{1pt}
\begin{picture}(150,100)
\put(0,0){\includegraphics[width=.07\textwidth]{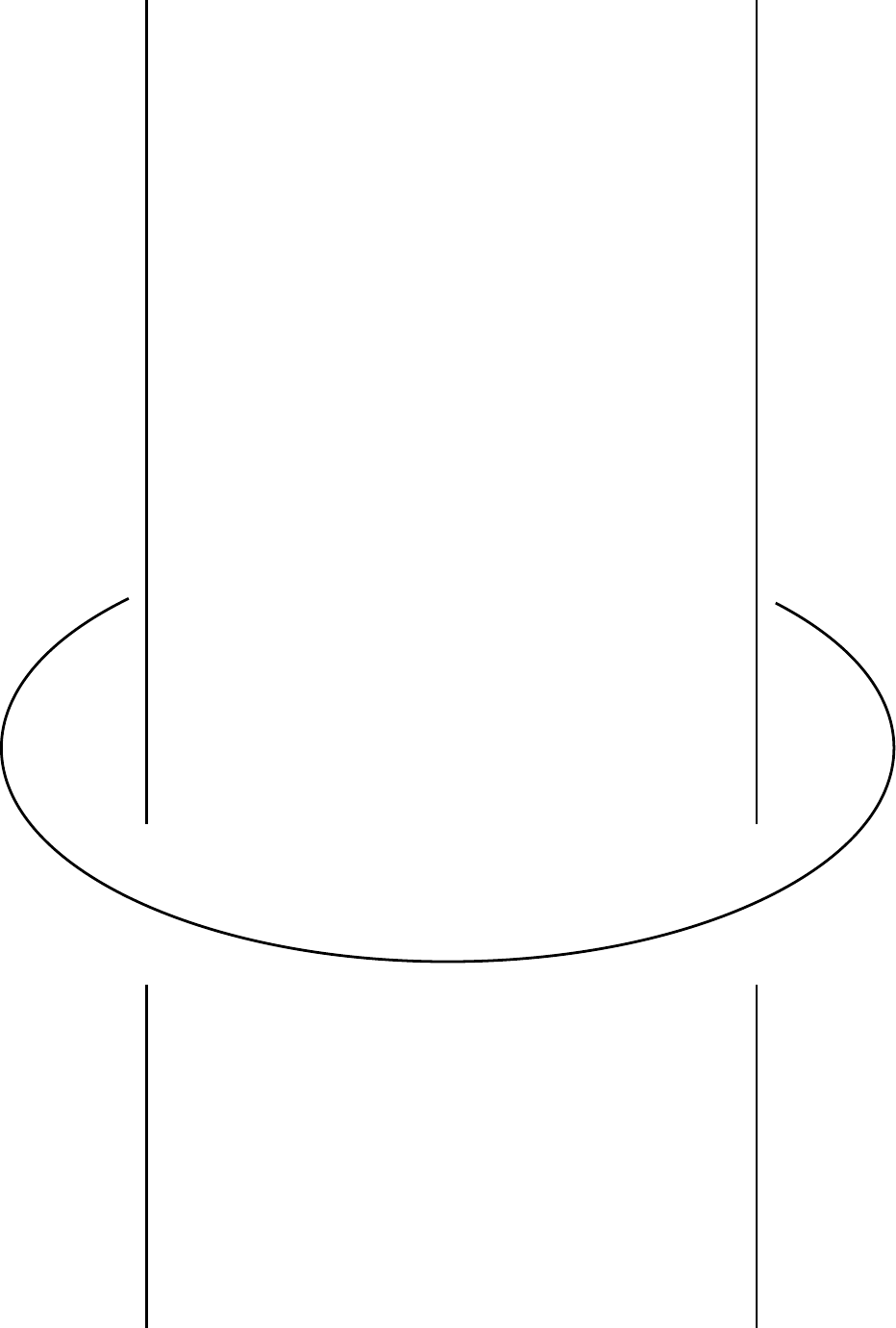}}
\put(90,0){\includegraphics[width=.1\textwidth]{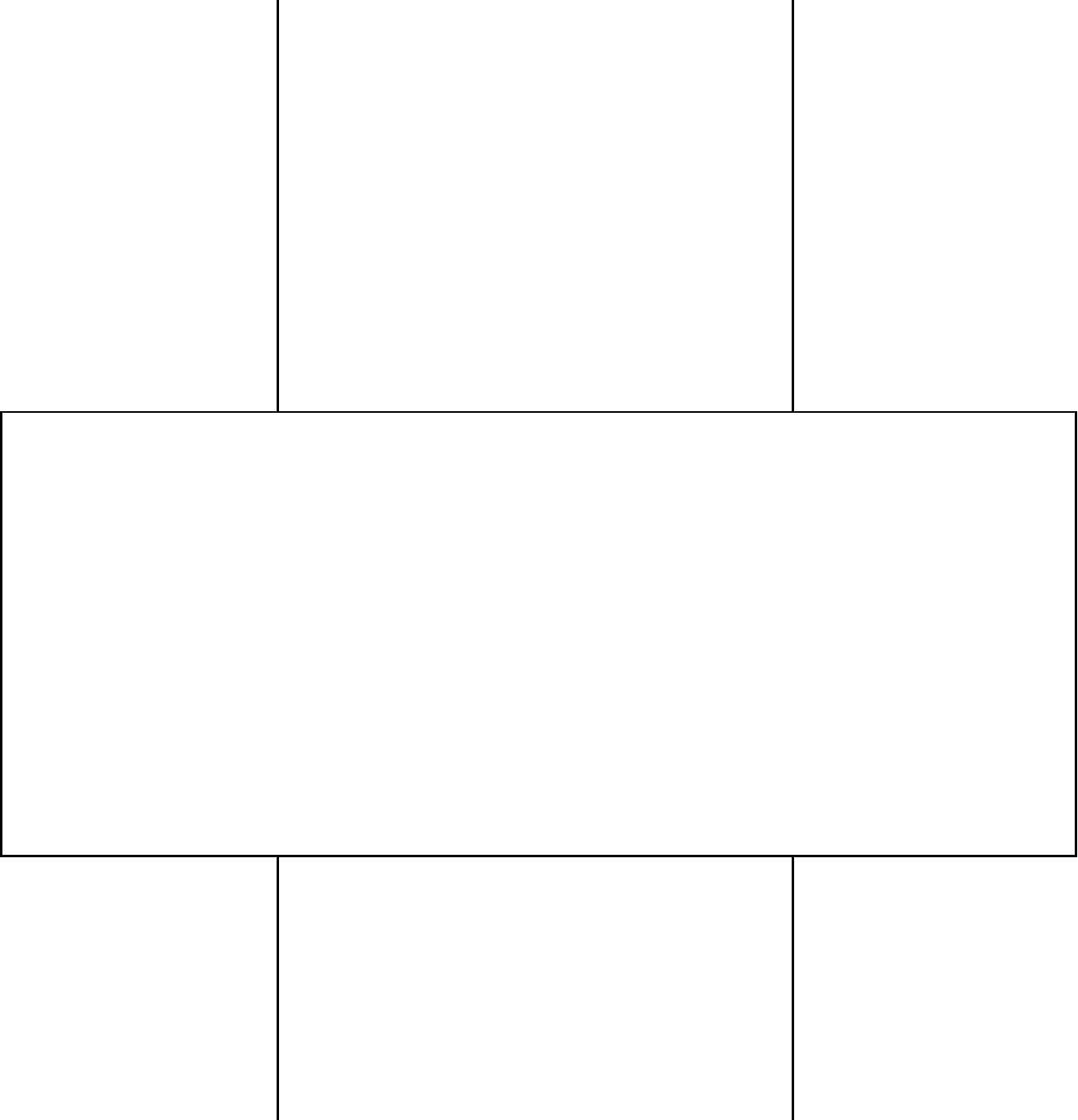}}
\put(15,20){$\eta$}
\put(110,15){$J$}

\end{picture}
\caption{Left: A portion of a knot $R$ going through a disk bounded by $\eta$.  Right: the knot $R_\eta(J)$}\label{fig:infection}
\end{figure}

Cochran, Harvey, and Leidy provide examples of robust doubling operators and prove the following theorem:

\begin{theorem*}[Theorem 7.7, \cite{CHL10}]
Let $I$ be an arbitrary indexing set.  Take $\{\mathcal{Q}_i=(q_{i,1}, \dots, q_{i,n})\}_{i\in I}$ to be a collection of $n$-tuples of polynomials which are termwise strongly coprime, that is, for all $i,j\in I$ and each $m$, $q_{i,m}$ and $q_{j,m}$ are strongly coprime.  Take $\{\mathcal{R}_i=(R^{i,n}_{\alpha_{i,n}}\circ\dots\circ R^{i,1}_{\alpha_{i,1}}\}_{i\in I}$ to be a set of iterated robust doubling operators such that $\Delta_{R_{i,m}}=q_{i,m}$.  For each $i\in I$, take $\mathcal{K}_i=\{K_{i,j}\}_{j=1}^\infty$ to be a sequence of knots with vanishing Arf-invariants, such that for each $i\in I$, the set of integrals of the Tristram-Levine signature functions $\{\rho_0(K_{i,j})\}_{j=1}^\infty$, is linearly independent of the first order signatures of $R_{i,1}$.  Then the set
$$
\{\mathcal{R}_i(K_{i,j})|~i\in I, 1\le j<\infty\}
$$
is linearly independent in $\F_n/\F_{n.5}$.
\end{theorem*}

From here on we refer to the integral of the Tristram-Levine signature function of a knot as its $\rho_0$-invariant.  For a discussion of von Neumann $\rho$-invariants, see \cite{COT03}.

The doubling operators that Cochran, Harvey, and Leidy construct, $R^{i,m}_{\alpha_{i,m}}$, satisfy all of the needed conditions and have only one non-zero first order signature.  By picking the sequences $\mathcal{K}_{i}$ to consist of knots with linearly independent $\rho_0$-invariants, the addition of this first order signature introduces at most one relation.  Thus, by removing at most one term from the sequence $\mathcal{K}_i$, the linear independence condition can be met.  

 Unfortunately, the nonzero first order signatures are difficult to compute, and so it is not yet known which $K_{i,j}$ must be removed from the sequence $\mathcal{K}_i$ to achieve this independence.  For a different set of robust doubling operators, we circumvent  this difficulty.
 
 Let $J_0, J_1,\ldots$ be a sequence of knots with vanishing Arf-invariant and linearly independent $\rho_0$-invariant.  \cite[Theorem 2.6]{COT04} provides such a sequence.  
 
 \begin{figure}[htbp]
\setlength{\unitlength}{1pt}
\begin{picture}(150,200)
\put(-90,0){\includegraphics[width=.35\textwidth]{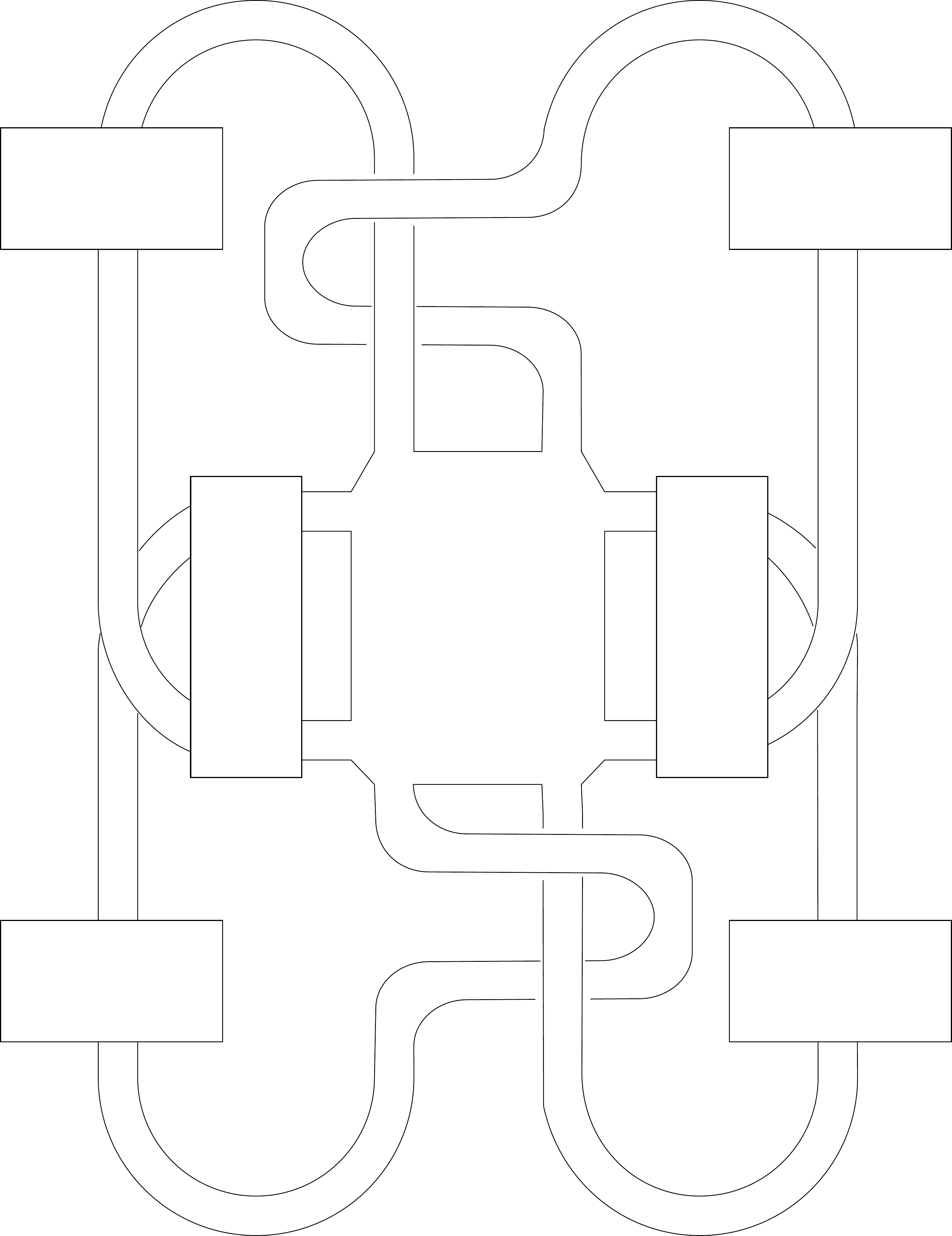}}
\put(90,0){\includegraphics[width=.35\textwidth]{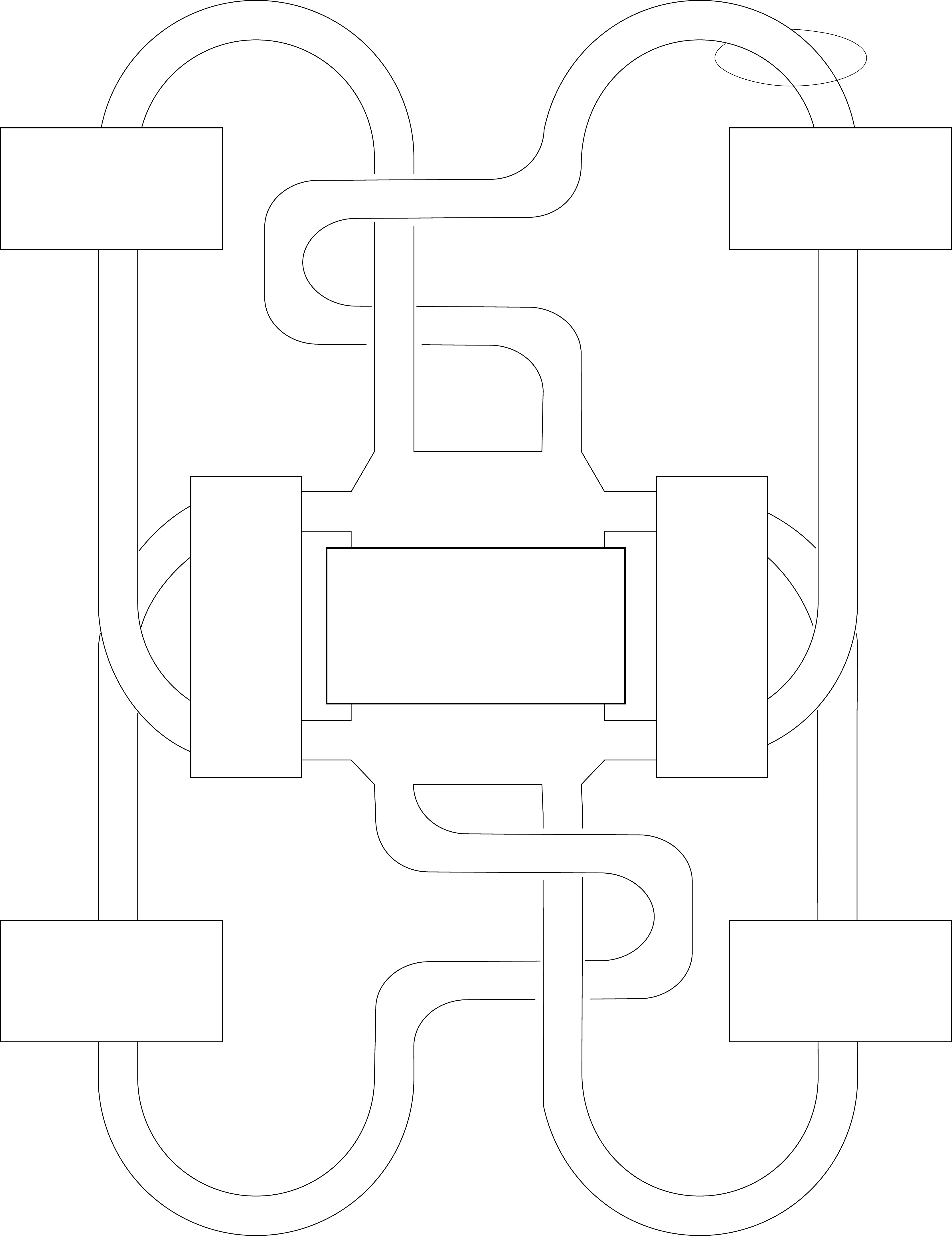}}
\put(225,190){$\eta$}
\put(160,95){$J_0$}
\put(-58,95){$+b$}
\put(15,95){$-b$}
\put(121,95){$+b$}
\put(195,95){$-b$}
\put(-88,162){$+b+1$}
\put(-88,37){$+b-1$}

\put(27,162){$-b+1$}
\put(27,37){$-b-1$}

\put(91,162){$+b+1$}
\put(91,37){$+b-1$}

\put(207,162){$-b+1$}
\put(207,37){$-b-1$}

\end{picture}
\caption{On the left: The ribbon knot $\widetilde{R}^b$.  On the right: the robust doubling operator $R^b_\eta$ gotten by infecting $\widetilde{R}^b$ by $J_0$.  The $\pm b$ indicates $b$ positive or negative full twists between bands without adding twists in either band.}\label{fig:slice knots}
\end{figure}
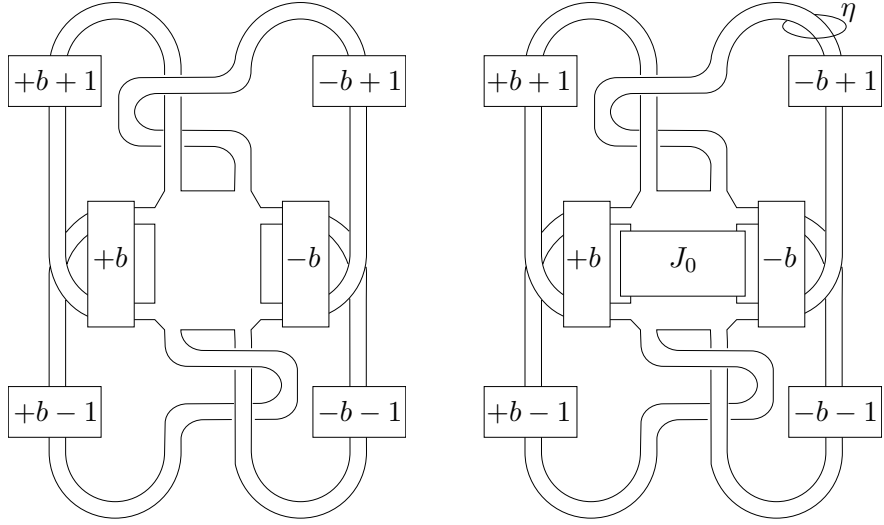

 For each positive integer $b$, consider the amphichiral ribbon knot $\widetilde{R}^b$ depicted on the left hand side of Figure~\ref{fig:slice knots}.  It has Alexander module $A_0(\widetilde{R}^b)=\frac{\Q[t,t^{-1}]}{\langle\delta_b(t)^2\rangle}$ where \[\delta_b(t)=bt^2-(1+2b)t+b.\]  This module has only two isotropic submodules, $P=\langle\delta_b\rangle$ and $0$.  The corresponding $\rho$-invariants both vanish: $\rho(\widetilde{R}^b, \phi_P)=0$ since $P$ corresponds to a ribbon disk, and $\rho(\widetilde{R}^b,\phi_0)=0$ by \cite[Proposition 4.5]{CHL08} since $\widetilde{R}^b$ is amphichiral.

 Of course this is problematic: the submodule $0$ cannot correspond to a ribbon disk.  No choice of $\eta$ will make $\widetilde{R}^b_\eta$ a robust doubling operator.  With that in mind we perform the infection by $J_0$ depicted on the right hand side of Figure~\ref{fig:slice knots}.    Calling the resulting ribbon knot $R^b$, we see that $\rho(R^b,\phi_P)=0$ since $R^b$ is still ribbon and $P$ corresponds to a ribbon disk for $R^b$.  By \cite[Proposition 3.2]{COT04}, $\rho(R^b,\phi_0)=\rho_0(J_0)$ which in particular is nonzero.  Thus, if $\eta$ is any unknotted curve representing a generator of $A_0(R^b)$ (such as the one depicted in Figure~\ref{fig:slice knots}), $R^b_\eta$ is a robust doubling operator
 
Notice that since $\delta_b = \Delta_{T_b}$ is exactly the Alexander polynomial of the $b$-twist knot, Theorem~\ref{examples of strong coprimality} asserts that the polynomials $\Delta_{R_a}$ and $\Delta_{R_b}$ are strongly coprime for $a\not= b$. 

Thus, we have an explicit set of knots fulfilling the conditions of \cite[Theorem 7.7]{CHL10}.  By applying the theorem, we get the following linearly independent sets in $\F_n/\F_{n.5}$:

\begin{theorem}   Let $J_0, J_1,\dots$ be a sequence of knots with vanishing Arf-invariant and linearly independent $\rho_0$-invariants. 
Let $A=\{(a_{i,1},\ldots, a_{i,n})\}$ be a collection of $n$-tuples of positive integers such that for all $i\neq j$ and $1\leq m\leq n$, we have $a_{i,m}\neq a_{j,m}$.  Given $a=(a_{i,1},\dots, a_{i,n})\in A$, let $\mathcal{R}_a$ be the iterated doubling operator $R^{a_{i,n}}_{\eta}\circ\dots\circ R^{a_{i,1}}_{\eta}$, where the $R^{a_{i,m}}_\eta$ are the robust doubling operators defined above.  Then the set
$$
\{\mathcal{R}_a(J_j)|~a\in A, 1\le j<\infty\}.
$$
is linearly independent in $\F_n/\F_{n.5}$.
\end{theorem}

\end{document}